\documentclass[11pt]{article}
\usepackage[a4paper,margin=1.0in]{geometry}
\usepackage[colorlinks,citecolor=magenta,linkcolor=black]{hyperref}
\pdfpagewidth=\paperwidth \pdfpageheight=\paperheight
\usepackage{amsfonts,amssymb,amsthm,amsmath,eucal,tabu,url}
\usepackage{fancyvrb}
\usepackage{pgf}
 \usepackage{array}
 \usepackage{tikz-cd}
 \usepackage{pstricks}
 \usepackage{pstricks-add}
 \usepackage{pgf,tikz}
 \usetikzlibrary{automata}
 \usetikzlibrary{arrows}
 \usepackage{indentfirst}
\usepackage{tabularx} 

\theoremstyle{plain}
\newtheorem{thm}{Theorem}[section]
\newtheorem{theorem}[thm]{Theorem}

\newtheorem{corollary}[thm]{Corollary}

\theoremstyle{definition}
\newtheorem{definition}[thm]{Definition}
\newtheorem{remark}[thm]{Remark}
\newtheorem{example}[thm]{Example}

\newtheorem{thevarthm}[thm]{\varthmname}

\newenvironment{varthm*}[1]{\trivlist\item[]{\bf #1.}\it}{\endtrivlist}

\theoremstyle{plain}
\newtheorem{custom}{{\rm Theorem}}

\renewcommand\geq{\geqslant}

\renewcommand\leq{\leqslant}

\newcommand\be{\begin{eqnarray*}}
\newcommand\ee{\end{eqnarray*}}

\newcommand\calo{{\mathcal O}}
\newcommand\calc{{\mathcal C}}

\newcommand\newop[2]{\def#1{\mathop{\rm #2}\nolimits}}
\newop\log{log}
\newop\ord{ord}
\newop\Gal{Gal}
\newop\SL{SL}
\newop\Bl{Bl}
\newop\mult{mult}
\newop\mass{mass}
\newop\div{div}
\newop\codim{codim}
\newop\sing{sing}
\newop\vdim{vdim}
\newop\edim{edim}
\newop\Ass{Ass}
\newop\size{size}
\newop\reg{reg}
\newop\satdeg{satdeg}
\newop\supp{supp}
\newop\Neg{Neg}
\newop\Nef{Nef}
\newop\Nefh{Nef_H}
\newop\Eff{Eff}
\newop\Zar{Zar}
\newop\MB{MB}
\newop\MBxC{MB\mathit{(x,C)}}
\newop\NnB{NnB}
\newop\Bigg{Big}
\newop\Effbar{\overline{\Eff}}

\def\keywordname{{\bfseries Keywords}}%
\def\keywords#1{\par\addvspace\medskipamount{\rightskip=0pt plus1cm
\def\and{\ifhmode\unskip\nobreak\fi\ $\cdot$
}\noindent\keywordname\enspace\ignorespaces#1\par}}
\def\subclassname{{\bfseries Mathematics Subject Classification
(2020)}\enspace}
\def\subclass#1{\par\addvspace\medskipamount{\rightskip=0pt plus1cm
\def\and{\ifhmode\unskip\nobreak\fi\ $\cdot$
}\noindent\subclassname\ignorespaces#1\par}}

\begin{document}
\title{Construction of free arrangements using point-line operators}
\author{Piotr Pokora and Xavier Roulleau}
\date{\today}
\maketitle

\thispagestyle{empty}
\begin{abstract}
We construct new examples of free curve arrangements in the complex projective plane using point-line operators recently defined by the second author. In particular, we construct a new example of a conic-line arrangement with ordinary quasi-homogeneous singularities that has non-trivial monodromy. 
\subclass{conic-line arrangements, plane curve singularities, monodromy, Lefschetz properties, unexpected curves, moduli spaces of hyperplane arrangements}
\end{abstract}
\section{Introduction}
The main aim of this paper is to present a new idea of constructing curve arrangements in the complex projective plane with particularly nice properties via the point-line operators recently introduced by the second author in \cite{XR1}.

Our motivation stems from several open questions in the theory of plane curves. Specifically, we are interested in problems related to constructing free plane curves, particularly in the context of the classical and Numerical Terao freeness conjecture \cite{CP, Terao}. From an algebraic combinatorics perspective, researchers attempt to construct hyperplane arrangements with many symmetries that are free, for instance by looking at reflection groups \cite{Terao1}. In combinatorics, interesting approaches to constructing families of free hyperplane arrangements involve examining the properties of intersection lattices. For example, it is known that arrangements with supersolvable intersection lattices are free \cite{Jambu}. In the context of curve arrangements, effective methods of constructing free curves include using pencils of curves \cite{King} or examining special classes of curve arrangements, for instance only with quasi--homogeneous singularities \cite{Schenck}. The main goal of our paper is to present a new method for constructing free plane curves using point-line operator techniques, denoted by  $\Lambda_{\mathfrak{n},\mathfrak{m}}$. We will introduce these techniques in more detail now. 

We consider a line arrangement $\mathcal{L} = \{\ell_{1}, \ldots, \ell_{n}\} \subset \mathbb{P}^{2}_{\mathbb{C}}$ and we fix $\mathcal{D}$ as the dual operator between $\mathbb{P}^{2}_{\mathbb{C}}$ and $\check{\mathbb{P}}^{2}_{\mathbb{C}}$, which to a given line arrangement $\mathcal{L}$ associates an arrangement of points, or more concretely the normals of the lines in $\mathcal{L}$. More precisely, for a line $\ell : ax + by + cz=0$ in $\mathcal{L}$ we have $\mathcal{D}(\ell) = (a:b:c) \in \check{\mathbb{P}}^{2}_{\mathbb{C}}$.

Fix a 
subset $\mathfrak{n} \subset \mathbb{Z}_{\geq 2}$, we define the operator $\mathcal{D}_{\mathfrak{n}}(\mathcal{L})$ that sends the line arrangement $\mathcal{L}$ to the line arrangement in the dual plane which is the union of the lines containing exactly $n$ points of $\mathcal{D}(\mathcal{L})$ for $n \in \mathfrak{n}$. We define our point-line operator as
$$\Lambda_{\mathfrak{n},\mathfrak{m}} =\mathcal{D}_{\mathfrak{n}}\circ \mathcal{D}_{\mathfrak{m}}.$$ 
For instance, if we apply $\Lambda_{\{2\},\{k\}}$ to a line arrangement $\mathcal{L}$, then the result is the union of lines that contain exactly $k$ double points of $\mathcal{L}$. It is worth pointing out that the result of applying $\Lambda_{\mathfrak{n},\mathfrak{m}}$ to $\calc$ for some choice of $\mathfrak{n}, \mathfrak{m}$ might be empty. 
\begin{example}
Let us consider an arrangement $\mathcal{L} \subset \mathbb{P}^{2}_{\mathbb{C}}$ consisting of $5$ generic lines in the plane, i.e., this arrangement has only $10$ double points as the intersections. If we now apply $\Lambda_{\{2\},\{3\}}$ to $\mathcal{L}$, then 
$$\Lambda_{\{2\},\{3\}}(\mathcal{L}) = \emptyset,$$
and this follows from the genericity assumption, so except lines in the arrangement that contain $4$ points from $\mathcal{D}(\mathcal{L})$, other lines are \emph{ordinary}, i.e., such an ordinary line contains exactly $2$ points from $\mathcal{D}(\mathcal{L})$.
\end{example}
For two integers $m,n\geq 2$, instead of $\Lambda_{\{n\geq 0\},\{m\geq 0\}}$, we 
will use the notation $\Lambda_{n,m}$, i.e.,  this is the operator that returns the line 
arrangement which is the union of the lines containing at least $m$ points of 
multiplicity at least $n$ in a given line arrangement $\mathcal{L}$.

For a line arrangement $\mathcal{L}$ and $k \in \mathbb{Z}_{\geq 2}$ let us denote by $n_{k} = n_{k}(\mathcal{L})$ the number of $k$-fold intersection points, i.e., points where exactly $k$ lines from the arrangement $\mathcal{L}$ meet, and for a point configuration $\mathcal{D}(\mathcal{L})$ and $r \in \mathbb{Z}_{\geq 2}$ we denote by $l_{r}$ the number of $r$-rich lines, i.e., lines in the dual plane $\check{\mathbb{P}}^{2}_{\mathbb{C}}$ that contain exactly $r$ points from $\mathcal{D}(\mathcal{L})$. 

In particular, $2$-rich lines are just ordinary lines. We should note that we will skip $\mathcal{L}$ when talking about $l_{r}$ and $n_{k}$ if it does not cause confusion.

We want to say a few words about the freeness of curves. Let $S := \mathbb{C}[x,y,z]$ be the graded polynomial ring and for a homogeneous polynomial $f \in S$ we define its Jacobian ideal as $$J_{f} = \langle \partial_{x}\, f, \partial_{y} \, f, \partial_{z} \, f \rangle.$$
\begin{definition}
\label{hom}
Let $C : f=0$ be a reduced curve in $\mathbb{P}^{2}_{\mathbb{C}}$ of degree $d$ given by $f \in S$. Denote by $M(f) := S/ J_{f}$ the associated Milnor algebra. We say that curve $C$ is $m$-syzygy if $M(f)$ has the following minimal graded free resolution:
$$0 \rightarrow \bigoplus_{i=1}^{m-2}S(-e_{i}) \rightarrow \bigoplus_{i=1}^{m}S(1-d - d_{i}) \rightarrow S^{3}(1-d)\rightarrow S \rightarrow M(f) \rightarrow 0$$

with $e_{1} \leq e_{2} \leq \ldots \leq e_{m-2}$ and $1\leq d_{1} \leq \ldots \leq d_{m}$. 
The $m$-tuple $(d_{1}, \ldots, d_{m})$ is called the set of exponents of $C$.
\end{definition}
\begin{definition}
We say that a reduced plane curve $C$ is \textbf{free} if and only if $C$ is $2$-syzygy and then we have $d_{1}+d_{2}=d-1$.
\end{definition}
Next, we need to recall the notion of the Alexander polynomial. Let $C : f=0$ be a reduced plane curve of degree $d$. Consider the complement $U:=\mathbb{P}^{2}_{\mathbb{C}}\setminus C$  and let $F : f=1$ be the corresponding Milnor fibre in $\mathbb{C}^{3}$ with the usual monodromy action 
$h : F \rightarrow F$. We consider the characteristic polynomials of the monodromy, namely
$$\triangle_{C}^{j}(t) = {\rm det}( t \cdot {\rm Id}-h^{j} \,| \, H^{j}(F,\mathbb{C})\}$$
for $j \in\{1,2,3\}$. When the curve $C$ is reduced, then $\triangle_{C}^{0}(t) = t-1$ and we have the following identity
\begin{equation}
\triangle_{C}^{0}(t)\triangle_{C}^{1}(t)\triangle_{C}^{2}(t) = (t^{d}-1)^{\chi(U)}.
\end{equation}
Here $\chi(U)$ is the Euler characteristic of the complement that can be computed as follows
$$\chi(U) = (d-1)(d-2) +1 -\mu(C),$$
where $\mu(C)$ is the total Milnor number of $C$. The polynomial $$\triangle(t) :=\triangle_{C}^{1}(t)$$
is called the Alexander polynomial of $C$.

Having our preparation done, we can sketch the main results of our paper, which are devoted to the constructions of new free arrangements of rational curves. It is worth recalling here that it is a notoriously difficult question to construct free arrangements, and especially when the number of irreducible components is expected to be large. We overcome this difficulty and show how to construct free arrangements consisting of many curves as irreducible components. Our first construction is based on the classical Hesse arrangement $\mathcal{H}$ consisting of $12$ lines with $n_{2}(\mathcal{H}) = 12$ and $n_{4}(\mathcal{H})=9$. Applying the operator $\Lambda_{2,2}$, i.e., this operator returns the line arrangement, which is the union of the lines containing at least two points of multiplicity at least two in a given arrangement $\mathcal{L}$, to the Hesse arrangement of lines we obtain the following. 
\begin{custom}
There exists a rigid arrangement $\mathcal{H}_{57}$ in the complex projective plane consisting of $57$ lines such that
$$n_{2} = 252, \quad n_{3} = 108, \quad n_{4} = 72, \quad n_{8}=21.$$
The arrangement $\mathcal{H}_{57}$ is free with exponents $(d_{1},d_{2}) = (25,31)$.\\
Moreover, the Alexander polynomial of $\mathcal{H}_{57}$
has the form
$$\triangle(t) = (t-1)^{56}.$$
\end{custom}
Here we say that a line arrangement is rigid if the moduli space of line arrangements with the same incidences between the lines is zero-dimensional. 

Next, we can apply our point-line operators to certain symmetric arrangements of lines with only double points. Let $\mathcal{C}_{8}$ be the arrangements consisting of lines determined by the sides of a regular octagon. Obviously our arrangement $\mathcal{C}_{8}$ consists of $8$ lines and $28$ double intersections. If we apply operator $\Lambda_{\{2\},3}$ to arrangement $\mathcal{C}_{8}$, i.e., if we take the union of the lines that are at least $3$-rich of double points, we obtain the following line arrangement $\mathcal{O}_{33}$.
\begin{custom}
There exists a rigid arrangement $\mathcal{O}_{33}$ in the real projective plane consisting of $33$ lines such that
$$n_{2} = 108, \quad n_{3} = 40, \quad n_{5} = 16, \quad n_{8}=5.$$
The arrangement $\mathcal{O}_{33}$ is free with exponents $(d_{1},d_{2}) = (15,17)$. 
\end{custom}
Moreover, we explain how to use the geometry of regular $n$-gons to construct further examples of free arrangements consisting of $49$ and $61$ lines. It is worth emphasizing that in our constructions using regular $n$-gons we obtain rigid arrangements, which means that the moduli spaces are zero dimensional. 

In the next step we construct a conic-line arrangement in the real projective plane by using the geometry of the indeterminacy locus of a certain point-line operator. As a result we obtain the following.
\begin{custom}
There exists an arrangement $\mathcal{CL}$ of $6$ lines and $6$ conics in the real projective plane such that
$$n_{2} = 12,\quad n_{6} = 9.$$
The arrangement $\mathcal{CL}$ is free with exponents $(d_{1}, d_{2}) = (4,13)$. Moreover, the Alexander polynomial of $\mathcal{CL}$ is non-trivial, i.e.,
$$\triangle(t) = (t^3+1)^{4}(t^{2}+t+1)^{4}(t-1)^{11}.$$
\end{custom}

We should also note here that we will discuss some additional interesting geometric properties determined by our newly constructed line arrangements, namely we will explain why point configurations dual to lines in arrangements $\mathcal{H}_{57}$ and $\mathcal{O}_{33}$ admit unexpected curves, and we will study the Strong Lefschetz Property for algebras determined by these arrangements.

\section{Constructions of curve arrangements}
Before we present proofs of our results, let us outline main techniques that we are going to use here. Our constructions use symbolic computations that are performed in \verb}SINGULAR}, \verb}MAGMA} and \verb}OSCAR}. We provide calculation scripts in both programs in \textbf{Ancillary File} to our paper that is available on \textbf{arXiv}. In order to check the freeness property, we compute the minimal free resolutions of the associated Milnor algebras, and this can be performed in any symbolic computations software. Now we would like to outline the way to compute the Alexander polynomials, which is an involving procedure. Since our arrangements are free, we can use \verb}SINGULAR} script provided in \cite{DimcaSc}. Let us present a short description of this procedure. Let
$$\alpha_{q} = {\rm exp}(-2\pi \iota q/d)$$
be a root of unity of order $d = {\rm deg}(C)$ with $0 \leq q \leq d$, and let us denote by $m(\alpha_{q})$ the multiplicity of $\alpha_{q}$ as a root of the Alexander polynomial $\triangle_{C}^{1}(t)$. One has $\alpha_{0} = \alpha_{d}=1$ and 
$$m(1) = b_{1}(U) = r-1,$$
where $r$ is equal to the number of irreducible components of $C$. Now, using script $\verb}monof3}(q_{1},q_{2})$ provided in \cite{DimcaSc}, we take $q_{1} = 3$, $q_{2}=d$, and as an output we get a table, where in the first column we have values of $q$ from $3$ to $d$, and the third column contains the data $n_{2}(q)$. Set $n_{2}(0)=n_{2}(1)=n_{2}(2)=0$, then we have the following identity
\begin{equation}
\label{multmo}
m(\alpha_{q}) = n_{2}(q) + n_{2}(d-q) \text{ for any } q \in \{0, \ldots , d\}.
\end{equation}
Using the above description, we can find our Alexander polynomial of $C$, and for more details regarding such computations, we refer to \cite{Symb}.

In order to study moduli spaces, and their rigidity, we need to recall basic results on matroids. 
\begin{definition}
A matroid $M$ consists of a finite set $E$ and a non-empty collection $\mathcal{B} \subset 2^{E}$ that satisfies the Steinitz exchange axiom, namely for each pair $A,B$ of distinct elements $\mathcal{B}$ and $x \in A \setminus B$ there is an element $y \in B \setminus A$ such that $(A\setminus\{x\})\cup \{y\} \in \mathcal{B}$.
\end{definition}
The elements in $\mathcal{B}$ are called bases and they have the same number $r$ of elements. Subsets of order $r$ of $E$ are called non-bases. Consider a vector configuration $v_1,\dots,v_n$ such that no two vectors are proportional, and consider the vectors as the columns of a $r \times n$ matrix $X$ -- we suppose $X$ has full-rank. The matroid for this configuration, denote here by $M[X]$, is the matroid whose ground set is $\{1, \ldots, n\}$ and its bases are the sets of $r$ columns that are of full-rank. We may view the columns as points in the projective space $\mathbb{P}^{r-1}$, so $X$ can be viewed as a geometric projective realization of $M$. 
\begin{definition}
Let $\mathbb{F}$ be a field. A matroid $M$ is $\mathbb{F}$-realizable if there exists a matrix $X$ with entries in the field $\mathbb{F}$ such that $M \cong M[X]$. If $M$ is $\mathbb{F}$-realizable for some field $\mathbb{F}$, then $M$ is said to be realizable. 
\end{definition}
Determining whether a matroid can be realized geometrically over a given field $\mathbb{F}$ is a very difficult problem. In this context, we can defined the following crucial object.
\begin{definition}
Given a field $\mathbb{F}$ and a matroid $M$, its realization space $\mathcal{R}(M;\mathbb{F})$ is a (possibly empty) algebraic variety (or scheme) defined over $\mathbb{F}$, whose closed points parametrize equivalence classes of point configurations in $\mathbb{P}^{r-1}_{\mathbb{F}}$ whose matroid is $M$, where we say that two configurations are equivalent if 
one can be transformed to the other by an element of ${\rm PGL}_{r}(\mathbb{F})$. In particular, matroid $M$ is $\mathbb{F}$-realizable if and only if $\mathcal{R}(M,\mathbb{F}) \neq \emptyset$.
\end{definition}
Now we explain how to construct the moduli spaces of line arrangements step by step. For this purpose, let $M=(\{1, \ldots ,n\}, \mathcal{B})$ be a simple matroid of rank $3$. A realization of $M$ over a given field $\mathbb{F}$ is a matrix $X 
\in \mathbb{F}^{3\times n}$ such that for all subsets $P \subset \{1, \ldots ,n\}$ of size $3$ we have
$$(\star) : \quad {\rm det} \, X_{P} \neq 0 \iff P \in \mathcal{B},$$
where $X_{P}$ is the $3\times 3$ submatrix consisting of the columns indexed by $P$. The kernels of the linear forms given by the columns of $P$ define an arrangement $\mathcal{L}$ of $n$ lines in $\mathbb{P}^{2}_{\mathbb{F}}$ whose intersection lattice is isomorphic to the lattice of flats of $M$. Observe that the condition $(\star)$ defines an ideal $I'$ in the ring $R\ = R[d]$, where
$$R = \mathbb{Z}[x_{ij} \, : \, i \in \{1,2,3\}, j \in \{1, \ldots , n\}],$$
given by
$$I' = \langle {\rm det}(X_{N}) \, : \, N \subset E \text{ is not a basis}, \, |N| = 3 \rangle + \bigg\langle 1 - d\prod_{B \in \mathcal{B}}{\rm det} X_{B} \bigg\rangle < R[d],$$
where $X = (x_{ij}) \in \mathbb{F}^{3\times n}$ is an $3\times n$ matrix having the variables $x_{ij}$ as the entries. Using the scheme-theoretic language, the realization space is an affine scheme that is described by 
$$\mathcal{R}(M;\mathbb{F}) := V(I') \subset \mathbb{A}^{3n+1}_{\mathbb{F}}={\rm Spec} \, R[d] \rightarrow {\rm Spec} \, \mathbb{Z}.$$
This description explains, heuristically, how to find the realization space, but this procedure, in practice, is rather involving. In order to find a realization space of a given line arrangement, we will follow the lines in \cite{CKS} and all the calculations can be done using \verb}OSCAR}. Using this approach, we are able to check that realization spaces of our line arrangements are zero-dimensional, and all necessary details regarding implementation of that procedure in our cases is described in the aforementioned \textbf{Ancillary File}.

After such a condensed introduction, we present proofs of our results.

Let us start with Theorem \textbf{A}.
\begin{proof}
Our starting point is the Hesse arrangement $\mathcal{H}$ consisting of $12$ lines which has $n_{2}=12$ and $n_{4}=9$ as the intersection points. Recall the defining equation of our arrangement $\mathcal{H}$:
\begin{multline*}
Q(x,y,z) = xyz(x+y+z)(x+y+e z)(x+y+e^{2} z)(x+e y +z)(x + e^{2} y +z)(e x + y + z) \\ (e^{2}x + y+z)(ex + e^{2}y + z)(e x + y + e^{2}z) 
\end{multline*}
where $e^{2}+e+1=0$. Now we apply the point-line operator $\Lambda_{2,2}$ to $\mathcal{H}$, i.e., this operator returns the line arrangement which is the union of the lines containing at least two points of multiplicity at least two in $\mathcal{H}$, so this is the less demanding variant among point-line operators. As a result of the action of this operator, we obtain a line arrangement $\mathcal{H}_{57}$ of $57$ lines with the intersection points 
$$n_{2} = 252, \quad n_{3} = 108, \quad n_{4} = 72, \quad n_{8}=21.$$
For the completeness of the paper, let us present the resulting line arrangement by providing the equations of lines:

\[
\begin{tabu}{llll}
\ell_{1}: & \, x + (-e - 1)y + ez  = 0, & \ell_{2}: & \, y + z = 0,    \\
\ell_{3}: & \, x - 2ey + ez = 0,        & \ell_{4}: & \, x + ey - 2ez  = 0, \\
\ell_{5}: & \, x + ey + (2e + 2)z = 0,  & \ell_{6}: & \, x + (-e - 1)z = 0, \\
\ell_{7}: & \, x + ey - 2z  = 0,        & \ell_{8}: & \, x - \frac{1}{2}ey - \frac{1}{2}ez = 0, \\
\ell_{9}: & \, x + y - 2z = 0,          &\ell_{10}: & \, x - \frac{1}{2}y - \frac{1}{2}z  = 0,  \\
\ell_{11}: & \, y + ez = 0,             &\ell_{12}: & \, x + (e + 1)y = 0, \\
\ell_{13}: & \, x + \frac{1}{2}(e + 1)y - \frac{1}{2}ez  = 0,    & \ell_{14}: & \, x + (2e + 2)y + ez  = 0, \\
\ell_{15}: & \, x + (2e + 2)y + (-e - 1)z = 0, &\ell_{16}: & \, x + ey  = 0, \\
\ell_{17}: & \, x - \frac{1}{2}y + \frac{1}{2}(e + 1)z  =0, & \ell_{18}: & \, z = 0, \\
\ell_{19}: & \, x + y + z  = 0,       & \ell_{20}: & \, x + y + ez  = 0, \\
\ell_{21}: & \, x - ez = 0, & \ell_{22}: & \, x + ey + ez  = 0, \\
\ell_{23}: & \, x - y = 0, & \ell_{24}: & \, x + ey + z  = 0, \\
\ell_{25}: & \, x + (-e - 1)y - 2ez = 0, & \ell_{26}: & \, x + ey + (-e - 1)z  = 0, \\
\ell_{27}: & \, x + (-e - 1)y + z = 0, & \ell_{28}: & \, x - 2y + ez  = 0, \\
\ell_{29}: & \, x + (-e - 1)y + (-e - 1)z = 0, & \ell_{30}: & \, x + y + (-e - 1)z  = 0, \\
\ell_{31}: & \, x - ey = 0, & \ell_{32}: & \, y  = 0, \\
\ell_{33}: & \, y + (e + 1)z = 0, & \ell_{34}: & \, x + (e + 1)z  = 0, \\
\ell_{35}: & \, x + y + (2e + 2)z = 0, & \ell_{36}: & \, x + \frac{1}{2}(e + 1)y - \frac{1}{2}z  = 0, \\
\ell_{37}: & \, x - \frac{1}{2}ey + \frac{1}{2}(e + 1)z = 0, & \ell_{38}: & \, y - ez  = 0, \\
\ell_{39}: & \, x - 2y + z = 0, & \ell_{40}: & \, x = 0, \\
\ell_{41}: & \, y - z = 0, & \ell_{42}: & \, x - z = 0, \\
\ell_{43}: & \, x + y = 0, & \ell_{44}: & \, x - 2ey + z = 0, \\
\ell_{45}: & \, x - 2y + (-e - 1)z = 0, & \ell_{46}: & \, x - 2ey + (-e - 1)z = 0, \\
\ell_{47}: & \, y + (-e - 1)z = 0, & \ell_{48}: & \, x + ez = 0, \\
\ell_{49}: & \, x + (2e + 2)y + z = 0, & \ell_{50}: & \, x + (-e - 1)y - 2z = 0, \\
\ell_{51}: & \, x + (-e - 1)y = 0, & \ell_{52}: & \, x + y - 2ez = 0, \\
\ell_{53}: & \, x + z = 0, & \ell_{54}: & \, x - \frac{1}{2}y - \frac{1}{2}ez = 0, \\
\ell_{55}: & \, x + (-e - 1)y + (2e + 2)z = 0, & \ell_{56}: & \, x + \frac{1}{2}(e + 1)y + \frac{1}{2}(e + 1)z = 0, \\
\ell_{57}: & \, x - \frac{1}{2}ey - \frac{1}{2}z = 0.&  & 
\end{tabu}
\]

Using the strategy presented above and symbolic computations in \verb}OSCAR}, we can check that the realization space of $\mathcal{H}_{57}$ is zero-dimensional, so our arrangement is rigid.

We verify the freeness of the arrangement. Using \verb}SINGULAR}, we can compute the minimal free resolution of the associated Milnor algebra, which is of the following form:
$$0\rightarrow S(-87) \oplus S(-81) \rightarrow S^{3}(-56) \rightarrow S,$$
and hence $\mathcal{H}_{57}$ is free with the exponents $(25,31)$. 

Using script \verb}mono3f} described in \cite{DimcaSc}, we can compute the Alexander polynomial of $\mathcal{H}_{57}$, namely
$$\triangle(t) = (t-1)^{56},$$
which means that the monodromy is trivial.
\end{proof}
\begin{remark}
It is worth recalling here that the Alexander polynomial of the Hesse arrangement $\mathcal{H}$ of $12$ lines has the form
$$\triangle(t) = (t-1)^{9}(t^{4}-1)^{2},$$ 
see \cite[Theorem 1.7]{DimMon}. 
This means that point-line operators do not automatically produce new line arrangements with non-trivial monodromy once we obviously start with arrangements with non-trivial monodromy. We should emphasize that line arrangements with non-trivial monodromy are very rare, and we have no global methods or simple criteria that can detect potential candidates for such line arrangements.
\end{remark}

Now we focus on arrangements that can be defined over the real numbers. Arrangements of lines defined over the reals attract the attention of many researchers working on combinatorial problems involving matroids and configurations. Here we focus only on the freeness property and present a construction based on a regular octagon. However, this idea can be further extended to regular decagons and dodecagons, see Remark \ref{Ocon} below.

We present here our proof of Theorem \textbf{B}.
\begin{proof}
This construction uses the geometric properties of regular $n$-gons. We start with the following regular octagon $\mathcal{C}_{8}$ which is given by the following defining equation:
\begin{multline*}
Q(x,y,z) = \bigg(x + (r - 1)y - z\bigg)\bigg(x + (r + 1)y + (-r - 1)z\bigg)\bigg(x + (-r - 1)y + (r + 1)z\bigg) \\
\bigg(x + (-r + 1)y + z\bigg)\bigg(x + (r - 1)y + z\bigg)\bigg(x + (r + 1)y + (r + 1)z\bigg)\\
\bigg(x + (-r - 1)y + (-r - 1)z\bigg) \bigg(x + (-r + 1)y - z\bigg). 
\end{multline*}
where $r^2-2=0$. When we say that our arrangement is a regular octagon, we mean that the lines are extensions of the sides of a regular octagon; all intersections are just double points and we have exactly $28$ such intersections. We apply $\Lambda_{\{2\},3}$ to arrangement $\mathcal{C}_{8}$, i.e., we take the union of the lines that are at least $3$-rich. As a result of this operation, we get the arrangement $\mathcal{O}_{33}$ consisting of $33$ lines and 
$$n_{2} = 108, \quad n_{3} = 40, \quad n_{5} = 16, \quad n_{8}=5,$$
which can be verified by \verb}SINGULAR}. 

We present below all the equations of $33$ lines building the arrangement $\mathcal{O}_{33}$.
\[
\small
\begin{tabu}{llll}
\ell_{1}: & \, x + (-r + 1)y - z  = 0,                                        &\ell_{2}: & \, x + (r + 1)y + (-2r - 3)z = 0,    \\
\ell_{3}: & \, x + y = 0,                                                     &\ell_{4}: & \, x + (-r - 1)y  = 0, \\
\ell_{5}: & \, x + (-r + 1)y = 0,                                             &\ell_{6}: & \, x + (r - 1)y + (r + 1)z= 0, \\
\ell_{7}: & \,   y = 0,                                                       &\ell_{8}: & \,  x + (-r + 1)y + (-r - 1)z = 0, \\
\ell_{9}: & \,   x + (r - 1)y + (r - 1)z = 0,                                 &\ell_{10}: & \,  x + (r + 1)y - z = 0,  \\
\ell_{11}: & \,   x + (r + 1)y + z= 0,                                        &\ell_{12}: & \,  x + (r + 1)y + (-r - 1)z = 0, \\
\ell_{13}: & \,   x + (r - 1)y - z = 0,                                       &\ell_{14}: & \, x + (-r - 1)y + (-r - 1)z= 0, \\
\ell_{15}: & \,   x + (r + 1)y + (2r + 3)z  = 0,                              &\ell_{16}: & \,  x + (r - 1)y + (-r - 1)z = 0, \\
\ell_{17}: & \,  x + (-r - 1)y + (-2r - 3)z = 0,                              &\ell_{18}: & \, x + (-r - 1)y + z = 0, \\
\ell_{19}: & \,  x + (-r + 1)y + (r - 1)z = 0,                                &\ell_{20}: & \,  x + (-r - 1)y + (r + 1)z= 0, \\
\ell_{21}: & \,  x + (r - 1)y + z = 0,                                        &\ell_{22}: & \, x + (-r + 1)y + z  = 0, \\
\ell_{23}: & \,  x = 0,                                                       &\ell_{24}: & \,  x + (r + 1)y + (r + 1)z = 0, \\
\ell_{25}: & \,  x + (r - 1)y = 0,                                            &\ell_{26}: & \,  x + (r + 1)y = 0, \\
\ell_{27}: & \,  x + (-r + 1)y + (-r + 1)z = 0,                               &\ell_{28}: & \,  x + (-r - 1)y - z = 0, \\
\ell_{29}: & \,  x - y = 0,                                                   &\ell_{30}: & \,  x + (r - 1)y + (-r + 1)z = 0, \\
\ell_{31}: & \,  z = 0,                                                       &\ell_{32}: & \,  x + (-r + 1)y + (r + 1)z = 0, \\
\ell_{33}: & \,  x + (-r - 1)y + (2r + 3)z = 0. &  & 
\end{tabu}\]
Having equations in hand, we can verify using \verb}OSCAR} that our arrangement is rigid, i.e., the realization space is zero-dimensional, and in the next step we can check the freeness property. Using \verb}SINGULAR}, we can compute the minimal free resolution of the associated Milnor algebra, which is of the following form:
$$0\rightarrow S(-49) \oplus S(-47) \rightarrow S^{3}(-32) \rightarrow S,$$
and hence $\mathcal{O}_{33}$ is free with the exponents $(15,17)$. 
\end{proof}
\begin{remark}
\label{Ocon}
We can construct more free line arrangements using point-line operators as follows.
For $n\geq 9$, consider a regular $n$-gon arrangement $\calc_n$ which is the union of the lines determined by sides of a regular $n$-gon. Define $k=n/2$ if $n$ is even and $k=(n-1)/2$ if $n$ is odd. 
As it was explained in \cite{KR}, the line arrangement $\calc=\Lambda_{\{2\},\, k-1}(\calc_n)$ is
the union of lines $\calc_n$, its $n$ lines of symmetries and the line at infinity. 
It is well-known that $\calc$ is a free simplicial line arrangement.
In the light of our discussion here, we get the line arrangement $\calo_{33}$ as $\calo_{33}=\Lambda_{\{2\}, \, k-1}(\calc_8)$, where $k=8/2=4$.
Using the same arguments as in Theorem \textbf{B}, we can show that the following line arrangements are also free:
$$\calo_{61}=\Lambda_{\{2\},\, 3}(\calc_{10}), \quad \quad  \calo_{49}=\Lambda_{\{2\},\, 4}(\calc_{12}),$$ 
where the lower index $m$ in $\mathcal{O}_{m}$ is the number of lines in our arrangement. 

Now we list singularities of these arrangements: for $\calo_{61}$ we have
$$n_{2} = 335, \quad n_{3} = 140, \quad n_{5} = 70, \quad n_{10}=1, \quad n_{12}=5,$$
and for $\calo_{49}$ we have
$$n_{2} = 204, \quad n_{3} = 96, \quad n_{4} = 6, \quad n_{5} = 24, \quad n_{6} = 6, \quad n_{7}=12, \quad n_{12}=1.$$
Moreover, similarly as in the case of $\mathcal{O}_{33}$, 
the line arrangements $\calo_{49}$ and $\calo_{61}$ are rigid. 
These arrangements can be defined over $\mathbb{Q}(\sqrt{5})$, $\mathbb{Q}(\sqrt{3})$, respectively. While regular a $12$-gon can be defined 
over $\mathbb{Q}(\sqrt{3})$, a regular $10$-gon is defined over a 
quadratic extension of $\mathbb{Q}(\sqrt{5})$. However, and somehow surprisingly, it is possible to find a projective transformation that sends it to a line 
arrangement defined over $\mathbb{Q}(\sqrt{5})$, and due to this reason $\calo_{61}$ 
can also be defined over $\mathbb{Q}(\sqrt{5})$. 

We tried to obtain other free line arrangements as images via point-line 
operators applied to regular $n$-gons $\calc_n$ with $n\leq 22$, but we could not find further examples. 
\begin{remark}
It was pointed out to us by Lukas K\"uhne that the arrangements $\mathcal{H}_{57}$ and $\mathcal{O}_{33}$ are not only free, but \textbf{divisionally free} -- see \cite{Abe} for all the necessary details on divisionally free arrangements.
\end{remark}
\end{remark}
Now we can pass to Theorem \textbf{C}.
\begin{proof}
In \cite{KR}, K\"uhne and the second author construct moduli spaces $\mathcal{R}$ of 
line arrangements $\mathcal{L}$ that are stable under some line operator $\Lambda$, i.e., 
if $\mathcal{L}$ is such a line arrangement, then $\Lambda(\mathcal{L})$ is a line arrangement
with the same combinatorics as $\mathcal{L}$, therefore there is an action of $\Lambda$ on $\mathcal{R}$. 
Among these moduli spaces, one is such that $\mathcal{R}$ is (birational 
to) $\mathbb{P}^2$, and then the action of $\Lambda$ on $\mathbb{P}^2$ is
by a rational self-map $\Lambda:\mathbb{P}^{2}\dashrightarrow \mathbb{P}^{2}$
given by {\footnotesize{}
\begin{equation}
\begin{array}{c}
\Lambda(x:y:z)=\bigg(-x^{6}-4x^{5}y-6x^{4}y^{2}-x^{4}yz-4x^{3}y^{3}-2x^{3}y^{2}z-2x^{3}yz^{2}-x^{2}y^{4}-x^{2}y^{3}z-3x^{2}y^{2}z^{2}\\
-x^{2}yz^{3}+x^{2}z^{4}-xy^{3}z^{2}-xy^{2}z^{3}\,:\,x^{6}+4x^{5}y-2x^{5}z+6x^{4}y^{2}-4x^{4}yz+4x^{3}y^{3}-2x^{3}y^{2}z\\
-4x^{3}yz^{2}+2x^{3}z^{3}+x^{2}y^{4}-6x^{2}y^{2}z^{2}+4x^{2}yz^{3}-x^{2}z^{4}-4xy^{3}z^{2}+2xy^{2}z^{3}-y^{4}z^{2}\,:\,\\
x^{6}+2x^{5}y+x^{4}y^{2}+x^{4}yz+2x^{3}y^{2}z-2x^{3}yz^{2}+x^{2}y^{3}z-3x^{2}y^{2}z^{2}\\
+x^{2}yz^{3}-x^{2}z^{4}-xy^{3}z^{2}+xy^{2}z^{3}-2xyz^{4}-y^{2}z^{4}\bigg).
\end{array}
\end{equation}
}
The indeterminacy points of $\Lambda$ are the following  $9$ points 
\begin{multline*}
 \mathcal{B} = \{(0:0:1), (-1:0:1), (-1:2:1), (1:0:1), (1:-2:1), (0:1:0), \\ (-1:1:0), (e^{2}:1:1), (e:1:1)\},   
\end{multline*}
where $e^{2}+e+1=0$.

There exists a unique pencil of cubics through these $9$ points, with $6$ 
degenerate fibers, each of which is the union of a conic and a line.
The arrangement $\mathcal{CL}\subset \mathbb{P}^{2}_{\mathbb{C}}$ that we are 
considering is the union of these lines and conics, and it is given by the following defining polynomial:
\begin{multline*}
    Q(x,y,z)= (x^{2} + 2xy + y^{2} - xz)(x^2 + xy + 2yz - z^2 )(x^2 + xz + yz)(x^2 + xy + z^2 )\cdot \\ (x^2 + 2xy - xz + yz)(x^2 - y^2 + xz + 2yz)(x + z)(2x + y)(x + y - z)y(x + y + z)(x - z).
\end{multline*}
The aforementioned pencil is generated by cubics and each cubic is the union of the $k^{th}$ line and the $k^{th}$ conic with $k\in \{1,\ldots,6\}$.

By the construction, the arrangement $\mathcal{CL}$ has the following intersection points
$$n_{2} = 12, \quad  n_{6}=9.$$
Our arrangement is of a pencil-type and all singularities of the arrangement are quasi-homogeneous, which follows from \cite{DJP}. Having the defining polynomial of $\mathcal{CL}$, we can compute the minimal free resolution of the associated Milnor algebra, namely
$$0\rightarrow S(-30) \oplus S(-21) \rightarrow S(-17)^3 \rightarrow S,$$
and hence $\mathcal{CL}$ is free with exponents $(d_{1}, d_{2}) = (4,13)$.

Using script \verb}mono3f} described in \cite{DimcaSc}, we can compute the Alexander polynomial of $\mathcal{CL}$. The output of the script is as follows (here we present only the information necessary for further calculations). 

\begin{center}
\footnotesize
\begin{BVerbatim} 
===============================
    q                 n_2(q)   
-------------------------------
    3                   0         
    4                   0         
    5                   0         
    6                   1        
    7                   0        
    8                   0        
    9                   2        
   10                   0        
   11                   0       
   12                   3        
   13                   0        
   14                   0        
   15                   4        
   16                   0        
   17                   0        
   18                  11        
\end{BVerbatim}
\end{center}
Based on the above table and using formula \eqref{multmo}, we can compute the Alexander polynomial, which has the form
$$\triangle(t) = (t^3+1)^{4}(t^{2}+t+1)^{4}(t-1)^{11},$$
and this shows that $\triangle(t)$ admits roots of unity of order $6$.
\end{proof}
\begin{remark}
The first example of a conic-line arrangement such that its Alexander polynomial admits $6$th roots of unity was presented in \cite{DPS}, and that example is constructed by using an Halphen pencil of index $2$. Our example fits into that picture, since $\mathcal{CL}$ is a pencil-type conic-line arrangement, but it has completely different geometric origins since in our case we have a pencil of cubics and the used Halphen pencil of index $2$ is generated by sextics. It is worth noticing that all known examples of conic-line arrangements with Alexander polynomials having roots of unity of order $6,7,8$ are constructed using suitable pencils of plane curves.
\end{remark}

\section{Unexpected curves and the Strong Lefschetz Property}
Now we would like to study some properties of point configurations determined by the duals of lines in arrangements $\mathcal{H}_{57}$ and $\mathcal{O}_{33}$. We start with the notion of unexpected curves that was introduced in \cite{Cook}.
\begin{definition}
Let $Z = \{p_{1}, \ldots, p_{d}\} \subset \mathbb{P}^{2}_{\mathbb{C}}$ be a finite set of mutually distinct points. We say that the set $Z$ admits an unexpected curve $C$ of degree $j\geq 2$ if
$$h^{0}(\mathbb{P}^{2}_{\mathbb{C}}, \mathcal{O}_{\mathbb{P}^{2}_{\mathbb{C}}}(j) \otimes \mathcal{I}(Z + (j-1)q)) > {\rm max} \bigg( 0, h^{0}(\mathbb{P}^{2}_{\mathbb{C}}(j) \otimes \mathcal{I}(Z))-\binom{j}{2}\bigg),$$
where $q$ is a generic point and $\mathcal{I}(Z + (j-1)q)$ is the ideal sheaf of functions vanishing along $Z$ and vanishing of order $j-1$ at $q$.
\end{definition}
Let us denote by $\mathcal{A}_{Z} \, : f_{Z}=0$ the line arrangement determined by the duals to points $Z$ and we define by $m(\mathcal{A}_{Z})$ the maximal multiplicity of an intersection point in $\mathcal{A}_{Z}$. In order to check whether our set of points $Z$ admits an unexpected curve, we can use the following result due to Dimca \cite{unexp}.

\begin{theorem}
The set of points $Z$ admits an unexpected curve if and only if 
$$m(\mathcal{A}_{Z}) \leq d_{1} +1 < \frac{d}{2},$$
where $d_{1}$ denotes the first exponent in the minimal resolution of the Milnor algebra associated with $\mathcal{A}_{Z}$. If these conditions are fulfilled, then $Z$ admits an unexpected curve of degree $j$ if and only if
$$d_{1} < j \leq d - d_{1}-2.$$
\end{theorem}
Using the above result, we deduce the following corollary.
\begin{corollary}
\begin{enumerate}
\item[a)] Let $Z_{57}$ be the set of points determined by the duals to lines in $\mathcal{H}_{57}$. Then $Z_{57}$ admits unexpected curves of degrees $j \in \{26, 27, 28, 29, 30\}$.
\item[b)] Let $Z_{33}$ be the set of points determined by the duals to lines in $\mathcal{O}_{33}$. Then $Z_{33}$ admits an unexpected curve of degree $j=16$.
\end{enumerate}
\end{corollary}
\begin{proof}
Recall that the arrangement $\mathcal{H}_{57}$ is free with exponents $(d_{1},d_{2}) = (25,31)$ and $m(\mathcal{H}_{57}) = 8$. Using our criterion, the set of points $Z_{57}$ admits an unexpected curve since

$$8=m(\mathcal{H}_{57}) \leq 26 < \frac{d}{2}= \frac{57}{2}.$$
Moreover, we can find admissible degrees of unexpected curves determined by $Z_{57}$, namely $j \in \{26, 27, 28, 29, 30\}$.

Recall also that the arrangement $\mathcal{O}_{33}$ is free with exponents $(d_{1},d_{2}) = (15,17)$ and $m(\mathcal{O}_{33}) = 8$. Using our criterion, the set of points $Z_{33}$ admits an unexpected curve since

$$8 = m(\mathcal{O}_{33}) \leq 16 < \frac{d}{2} = \frac{33}{2} .$$
Moreover, the only admissible degree of an unexpected curve determined by $Z_{33}$ is $j = 16$.
\end{proof}

Now we pass to the Strong Lefschetz property.
\begin{definition}
An artinian algebra $A = S/I$ satisfies the strong Lefschetz property (SLP) at range $k$ in degree $d$ if, for a general linear form $L$, the homomorphism 
$$\times L^{k} \, : [A]_{d} \rightarrow [A]_{d+k}$$
has maximal rank. Otherwise, we say that $A$ fails the SLP at range $k$ in degree $d$.
\end{definition}

In that context, one has the following result \cite[Theorem 7.5]{Cook} 

\begin{theorem}
Let $\mathcal{A}_{Z} \, :  f_{Z}=0$ be a line arrangement in $\mathbb{P}^{2}_{\mathbb{C}}$, where $f_{Z} = \ell_{1} \cdots \ell_{d}$, and let $Z$ be the set of points in $\mathbb{P}^{2}_{\mathbb{C}}$ dual to these lines. Then $Z$ admits an unexpected curve of degree $j+1$ if and only if $S/(\ell_{1}^{j+1}, \ldots, \ell_{d}^{j+1})$ fails the SLP in range $2$ and degree $j-1$.
\end{theorem}
From that perspective, we have the following corollary.
\begin{corollary}
\begin{enumerate}
\item[a)] The algebra $A = S/(\ell_{1}^{j+1}, \ldots , \ell_{57}^{j+1})$ associated with the arrangement $\mathcal{H}_{57}$ fails the SLP in range $2$ and degree $j-1$ with $j \in \{25,26,27,28,29\}$.
\item[b)] The algebra $A = S/(\ell_{1}^{16}, \ldots , \ell_{33}^{16})$ associated with the arrangement $\mathcal{O}_{33}$ fails the SLP in range $2$ and degree $14$.
\end{enumerate}
\end{corollary}

\section*{Acknowledgement}
The idea behind this paper originated when both authors visited the Max Planck Institute for Mathematics in Bonn, and we are very grateful for the excellent working conditions and hospitality.

We would like to thank Jean Vall\`es and Lukas K\"uhne for useful comments.

Piotr Pokora is supported by the National Science Centre (Poland) Sonata Bis Grant  \textbf{2023/50/E/ST1/00025}. For the purpose of Open Access, the author has applied a CC-BY public copyright licence to any Author Accepted Manuscript (AAM) version arising from this submission.

\vskip 0.5 cm
\bigskip
Piotr Pokora,
Department of Mathematics,
University of the National Education Commission Krakow,
Podchor\c a\.zych 2,
PL-30-084 Krak\'ow, Poland. \\
\nopagebreak
\textit{E-mail address:} \texttt{piotr.pokora@uken.krakow.pl}

\bigskip
Xavier Roulleau,
Universit\'e d'Angers, 
CNRS, LAREMA, SFR MATHSTIC, 
F-49000 Angers, France. \\
\nopagebreak
\textit{E-mail address:} \texttt{xavier.roulleau@univ-angers.fr}

\begin{thebibliography}{000}

\bibitem{Abe}
T. Abe, Divisionally free arrangements of hyperplanes. \textit{Invent. Math.} \textbf{204(1)}: 317 -- 346 (2016).
\bibitem{Cook}
D. Cook II, B. Harbourne, J. Migliore, and U. Nagel, Line arrangements and configurations of points with an unexpected geometric property. \textit{Compos. Math.} \textbf{154(10)}: 2150 -- 2194 (2018).

\bibitem{CKS}
D. Corey, L. K\"uhne, and B. Schröter, Matroids. A chapter in \textit{The computer algebra system OSCAR. Algorithms and examples.} Algorithms and Computation in Mathematics 32. Decker, Wolfram (ed.); Eder, Christian (ed.); Fieker, Claus (ed.); Horn, Max (ed.); Joswig, Michael (ed.) Cham: Springer (ISBN 978-3-031-62126-0/hbk; 978-3-031-62129-1/pbk; 978-3-031-62127-7/ebook). xviii, 491 p. (2025).

\bibitem{CP}
M. Cuntz and P. Pokora, Singular plane curves: freeness and combinatorics. \textit{Innov. Incidence Geom.} \textbf{22(1)}: 47 -- 65 (2025).

\bibitem{DimMon}
A. Dimca, On the Milnor monodromy of the irreducible complex reflection arrangements. \textit{J. Inst. Math. Jussieu} \textbf{18(6)}: 1215 -- 1231 (2019).

\bibitem{unexp}
A. Dimca, Unexpected curves in $\mathbb{P}^{2}$, line arrangements, and minimal degree of Jacobian relations. \textit{J. Commut. Algebra} \textbf{15(1)}: 15 -- 30 (2023).

\bibitem{DJP}
A. Dimca, M. Janasz, and P. Pokora, On plane conic arrangements with nodes and tacnodes. \textit{Innov. Incidence Geom.} \textbf{19(2)}: 47 -- 58 (2022).

\bibitem{DPS}
A. Dimca, P. Pokora, and G. Sticlaru, On the Alexander polynomials of conic-line arrangements. Accepted for publication in \textit{Ann. Sc. Norm. Super. Pisa, Cl. Sci. (5)}.

\bibitem{Symb}
A. Dimca, G. Sticlaru, Computing the monodromy and pole order filtration on Milnor fiber cohomology of plane curves. \textit{J. Symb. Comput.} \textbf{91}: 98 -- 115 (2019).

\bibitem{DimcaSc}
A. Dimca and G. Sticlaru, Milnor monodromy of free curves in the projective plane. \url{https://math.univ-cotedazur.fr/~dimca/monoFREE3.pdf}.

\bibitem{Jambu}
M. Jambu and H. Terao, Free arrangements of hyperplanes and supersolvable lattices. \textit{Adv. Math.} \textbf{52}: 248 -- 258 (1984).
\bibitem{King}
W. K. King and J. Vall\`es, New examples of free projective curves. \textit{Rend. Ist. Mat. Univ. Trieste} \textbf{54}: Paper No. 13, 17 p. (2022).
\bibitem{KR}
L. K\"uhne and X. Roulleau, Regular polygons, line operators, and elliptic modular surfaces as realization spaces of matroids. \textit{Int. Math. Res. Not.} \textbf{2025(24)}: Article ID rnaf367, 18 p. (2025).

\bibitem{XR1}
X. Roulleau,  On some operators acting on line arrangements and their dynamics. \textit{Enseign. Math. (2)} \textbf{72(1)}: 47 -- 84 (2026).

\bibitem{Schenck}
H. Schenck, H. Terao, and M. Yoshinaga, Logarithmic vector fields for curve configurations in $\mathbb{P}^{2}$ with quasihomogeneous singularities. \textit{Math. Res. Lett.} \textbf{25(6)}: 1977 -- 1992 (2018).

\bibitem{Terao}
H. Terao, Arrangements of hyperplanes and their freeness I. \textit{J. Fac. Sci., Univ. Tokyo, Sect. I A} \textbf{27}: 293 -- 312 (1980).

\bibitem{Terao1}
H. Terao, Generalized exponents of a free arrangement of hyperplanes and Shepherd- Todd-Brieskorn formula. \textit{Invent. Math.} \textbf{63}: 159 -- 179 (1981).



\end{thebibliography}
\end{document}